\documentclass{article}   	
\usepackage{geometry}               
\usepackage{graphicx}
\usepackage[usenames,dvipsnames]{xcolor}
\usepackage{amssymb,amsmath,amsthm,amsfonts}
\usepackage{bm,cancel}                     
\usepackage[english]{babel}
\usepackage[latin1]{inputenc}
\usepackage[colorlinks]{hyperref}
\hypersetup{linkcolor=blue,citecolor=blue,filecolor=black,urlcolor=blue}

\usepackage[sc]{mathpazo}

\newtheorem{proposition}{Proposition}[section]
\newtheorem{theorem}[proposition]{Theorem}
\newtheorem{corollary}[proposition]{Corollary}
\newtheorem{lemma}[proposition]{Lemma}

\theoremstyle{definition}

\theoremstyle{remark}
\newtheorem{remark}[proposition]{Remark}

\theoremstyle{open}
\newtheorem{open}[proposition]{Open Problem}

\numberwithin{equation}{section}

\newcommand{\N}{{\mathbb{N}}}

\newcommand{\R}{{\mathbb{R}}}

\newcommand{\beq}{\begin{equation}}
\newcommand{\eeq}{\end{equation}}
\newcommand{\ben}{\begin{enumerate}}
\newcommand{\een}{\end{enumerate}}
\newcommand{\bit}{\begin{itemize}}
\newcommand{\eit}{\end{itemize}}

\newcommand{\capu}[1][\alpha]{{\partial_t^{#1}}}

\title{Time-fractional equations with reaction terms: fundamental solutions
and asymptotics}

\author{Serena Dipierro\thanks{Department
of Mathematics and Statistics,
University of Western Australia,
35 Stirling Highway,
Crawley WA 6009, Australia. {\tt serena.dipierro@uwa.edu.au} },
Benedetta Pellacci\thanks{Dipartimento di Matematica
e Fisica, Universit\`a della Campania ``Luigi Vanvitelli'',
Viale Lincoln 5,
81100 Caserta, Italy. {\tt benedetta.pellacci@unicampania.it}},
Enrico Valdinoci\thanks{Department of Mathematics and Statistics,
University of Western Australia,
35 Stirling Highway,
Crawley WA 6009, Australia. {\tt enrico.valdinoci@uwa.edu.au}}, and
Gianmaria Verzini\thanks{Dipartimento di Matematica, Politecnico di Milano,
Piazza Leonardo da Vinci 32, 20133 Milano,
Italy. {\tt gianmaria.verzini@polimi.it}}}

\begin{document}
\maketitle


\begin{abstract}
We analyze the fundamental solution of a time-fractional problem, establishing existence
and uniqueness in an appropriate functional space.

We also focus on the one-dimensional spatial setting in the case in which the time-fractional
exponent is equal to, or larger than, $\frac12$. In this situation, we prove that
the speed of invasion of the fundamental solution is at least ``almost of square root type'',
namely it is larger than~$ct^\beta$ for any given~$c>0$ and~$\beta\in\left(0,\frac12\right)$.
\end{abstract}
\bigskip

\noindent{\bf Keywords:} {\em
Time fractional diffusion, heat conduction with memory, Caputo fractional derivatives, fundamental solution, asymptotic behavior of solution.}
\medskip

\noindent{\bf Mathematics Subject Classification:}
35R11, 35C15, 35B40, 35K57, 35K08, 26A33.
\bigskip

\section{Introduction}

In this note we will consider a parabolic problem
with time-fractional diffusion.
The spatial diffusion will be modeled by the Laplacian operator,
while the time derivative is of fractional type.
In particular, we consider the so-called Caputo
derivative of order~$\alpha\in(0,1)$, with initial time~$t=0$,
given by
\begin{equation}\label{CAPO} \capu u(t):=\frac{1}{\Gamma(1-\alpha)}\int_0^t\frac{\dot u(\tau)}{(t-\tau)^\alpha}\,d\tau.\end{equation}
We consider the evolution problem
\begin{equation}\label{eq:lin_reac_diff}
\begin{cases}
\capu u-d\Delta u=au, \qquad\qquad& x\in\R^N,\ t>0\\
u(x,0)=\delta,
\end{cases}
\end{equation}
where the constants $d,a$ satisfy $d>0$, $a\ge0$, and the $N$-dimensional Dirac delta distribution $\delta$ is centered at
$x = 0$.
\medskip

Time-fractional equations such as the one in~\eqref{eq:lin_reac_diff}
are very intriguing from the pure mathematical point of view,
as they offer a great technical challenge to the classical methods
of ordinary and partial differential equations, since many of
the standard techniques based on explicit barriers, maximum principle
and bootstrap differentiations simply do not work in this setting.
Furthermore, time-fractional diffusion naturally arises in a number
of real-world phenomena in which ``anomalous'' diffusion takes
place in view of a ``memory effect'' which is mathematically
described by the polynomial kernel in~\eqref{CAPO}.

As a matter of fact, the original motivation for the time-fractional operator
in~\eqref{CAPO} arose in~\cite{MR2379269} with the goal of describing
in mathematical terms an initial value problem arising in geophysics, to be confronted
with experimental data.

Other powerful applications of the operator in~\eqref{CAPO}
occur in the theory of viscoelastic fluids: roughly speaking,
in this context one models the viscoelastic effects
as an ideal superposition of ``purely elastic''
terms, governed by Hooke's Law
(relating the force to the displacement), and ``purely viscous'' terms,
governed by Newton's Law (in which
forces are related instead to velocities). Such a superposition
of zero-order effects (due to Hooke's Law) and first-order ones
(due to Newton's Law) are often conveniently modeled in terms
of fractional derivatives, see e.g.
Section~10.2 in~\cite{MR1658022} and the references therein.

Interestingly, this viscoelastic effect and its relation with fractional
calculus can be also explicitly understood in terms of mechanical systems
of springs and dampers, producing the operator in~\eqref{CAPO}
in a rigorous asymptotic way from fundamental physics considerations,
see~\cite{MR1890111} for additional details on these models and on related problems.

Furthermore, fractional diffusive operators as in~\eqref{CAPO}
and time-fractional heat equations as in~\eqref{eq:lin_reac_diff}
also arise in classical models when a complicated or fractal geometry
of the environment comes into play. In particular, the classical
heat equation on a very ramified comb structure naturally leads
to a time-fractional equation, see~\cite{COMB} and~\cite{MR3856678}.
The diffusive properties on highly ramified networks
have also fundamental consequences in neurology, since
this type of anomalous diffusion has been experimentally measured
in neurons, see e.g.~\cite{SANTA} and the references therein.
See~\cite{MR3645874} for a recent review on this and related topics,
and also~\cite{MR3814763} for a concrete mathematical model.

A number of natural applications of fractional derivatives
also occur in statistics, mechanics, engineering and finance, see e.g.~\cite{2017arXiv171011567A}
and~\cite{CARBOTTI}
for a series of concrete examples and further motivations.
We also refer to~\cite{MR2624107}, \cite{MR0444394}, \cite{MR860085},
\cite{MR1251624}, and
\cite{FERR} for historical introductions to the fractional calculus.
\medskip

The results that we obtain in this paper
are the following. First of all, we prove an {\em
existence and uniqueness} theory for the fundamental solution
in~\eqref{eq:lin_reac_diff} in the appropriate functional spaces, in any dimension.
Then, we specialize to the case of dimension~$1$, with fractional exponent
greater or equal than~$\frac12$, and we establish
precise asymptotics for the fundamental solution, in particular
by determining that the ``rate of invasion'' is faster than~$\{x=ct^\beta\}$
for any given~$c>0$ and~$\beta\in\left(0,\frac12\right)$.\medskip

To obtain these results, the precise mathematical framework in which we work is the following.
We consider the Sobolev
space~$H^{m}(\R^N)$
and denote by~$H^{-m}(\R^N)$
the dual of~$H^{m}(\R^N)$. For concreteness,
we will take~$\N\ni m>N/2$ (in this way,
functions in~$H^{m}(\R^N)$ are necessarily
continuous, and accordingly
the Dirac delta belongs to~$H^{-m}(\R^N)$).

Given~$v\in H^{-m}(\R^N)$, we let
$$ \| v\|_{H^{-m}(\R^N)}:= \sup_{ \varphi\in H^{m}(\R^N)\setminus \{0\} }
\frac{\big| \langle v,\varphi\rangle\big|}{\| \varphi\|_{H^{m}(\R^N)}}.$$
Then,
we say that~$v\in L^\infty_\alpha\big([0,T], H^{-m}(\R^N)\big)$ if, for all~$t\in[0,T]$,
we have that~$v(t)\in H^{-m}(\R^N)$ and
\begin{equation}\label{Dedf2}
\sup_{t\in[0,T]} t^\alpha \| v(t)\|_{H^{-m}(\R^N)}<+\infty.
\end{equation}
We also recall that, for all~$\alpha\in(0,1)$, the time-fractional equation
$$ \partial^\alpha_t u(t)=f(t)$$
can be reduced to the Volterra integral equation
\begin{equation}\label{VOLTERRA}
u(t)=u(0)+\frac{1}{\Gamma(\alpha)}\int_0^t \frac{f(\tau)}{(t-\tau)^{1-\alpha}}\,d\tau ,\end{equation}
see e.g.~\cite[Lemma 6.2]{di}.
Therefore, we define~$u$ to be a distributional
solution (or, briefly, a solution) of~\eqref{eq:lin_reac_diff}
if
\begin{equation*}
u\in L^\infty_\alpha\big([0,T], H^{-m}(\R^N)\big)\end{equation*} and
\begin{equation}\label{DEdebo}
\langle u(t),\varphi\rangle=\varphi(0)+\frac{1}{\Gamma(\alpha)}\int_0^t \frac{
d\langle u(\tau),\Delta\varphi \rangle+a\langle u(\tau),\varphi\rangle
}{(t-\tau)^{1-\alpha}}\,d\tau,
\end{equation}
for all~$\varphi\in C^\infty_0(\R^N)$.\medskip

In this setting, we have that

\begin{theorem}\label{EXUN}
There exists a unique solution of~\eqref{eq:lin_reac_diff}.
\end{theorem}

Without the reaction term, i.e., when~$a=0$ in~\eqref{eq:lin_reac_diff},
the fundamental solution of the fractional heat equation, with an initial datum in a suitable Lebesgue space, has been studied
in detail in~\cite{MR1829592}.

Interestingly,
one can give an explicit representation
of the solution of~\eqref{eq:lin_reac_diff} in terms
of special functions. To this end,
it is convenient to exploit the Mittag-Leffler function
\begin{equation}\label{eq:defML}
E_{\alpha}(r):=\sum_{k=0}^{\infty} \frac{r^k}{\Gamma(1+k\alpha)}.
\end{equation}
Although the above series converges in more general situations,
we will mainly deal with the cases
$\alpha\in(0,1]$, $r\in\R$.
In this setting, the fundamental solution provided by Theorem~\ref{EXUN}
can be written in the form
\begin{equation}\label{MIT}
u(x,t) :={\mathcal{F}}^{-1} \Big(E_{\alpha}\big((a-4\pi^2
d|\xi|^{2})t^{\alpha}\big)\Big),
\end{equation}
where~${\mathcal{F}}$ denotes the Fourier Transform and~${\mathcal{F}}^{-1}$
its inverse.

We observe that while it is somehow straightforward to ``guess'' that~\eqref{MIT}
is ``the'' solution of~\eqref{eq:lin_reac_diff} at a ``formal'' level of
linear calculations, some care is needed
to establish a coherent existence and uniqueness theory in the appropriate
functional spaces, and this is indeed the core of Theorem~\ref{EXUN}. In particular,
a powerful tool which can not be applied in a direct way here is the abstract theory of uniqueness and correctness classes for Cauchy
problems in topological vector spaces, as developed by Gel'fand and Shilov \cite{MR0435833}.
Indeed, such theory strongly relies on the validity of the Leibniz rule for the differentiation of a product (or, better, of a duality pairing), which fails in the time-fractional case.

Furthermore, the representation formula in~\eqref{MIT} reveals an important
structural difference with respect to the classical case. Indeed,
for the standard heat equation, the Fourier Transform of the fundamental
solution is a Gaussian function, thus possessing nice smoothness and decaying
properties. Instead, formula~\eqref{MIT} highlights
the memory effect introduced by the time-fractional operator, which
produces in this setting a significant loss of regularity in terms
of functional spaces. As a matter of fact, from~\eqref{MIT}
and some well-established properties of the Mittag-Leffler function (see e.g.
formula~\eqref{eq:ML_asynt-infty} here below),
it follows that
\begin{equation}\label{LPQUA0} {\mathcal{F}}u(\xi,t)\simeq \frac{1}{\Gamma(1-\alpha)\;
(4\pi^2d|\xi|^{2}-a)t^{\alpha}}\qquad{\mbox{ as }}|\xi|\to+\infty.\end{equation}
Accordingly, for a given~$t>0$, we have that
\begin{equation}\label{LPQUA} {\mathcal{F}}u(\cdot,t)\in L^p(\R^N) \qquad{\mbox{ if and only if }}\qquad
p\in\left(\frac{N}2,+\infty\right),\end{equation}
which is an important difference with the classical case in which the fast
decay at infinity implies integrability of any order.

Furthermore, from~\eqref{LPQUA} it follows that when~$N=1$
we have that~${\mathcal{F}}u(\cdot,t)\in L^1(\R)\cap L^2(\R)$.
Instead, when~$N\in\{2,3\}$, we have that~${\mathcal{F}}u(\cdot,t)\in L^2(\R^N)\setminus L^1(\R^N)$,
and, when~$N\ge4$, we have that~${\mathcal{F}}u(\cdot,t)\not\in L^2(\R^N)$. In particular,
if~$N\ge4$, it is not even clear whether $u$ belongs to some Lebesgue space, or it is merely a
tempered distribution.
{F}rom~\eqref{LPQUA0}, one also sees that
$$ {\mbox{ if $N\in\{2,3\}$ then $u(0,t)=+\infty$ for all~$t>0$,}}$$
which is also an important difference with respect to the classical case:
namely, differently from the standard diffusion equation, in higher dimension the memory effect
of equation~\eqref{eq:lin_reac_diff} {\em persists at the origin for all times},
preventing any decay whatsoever from the initial singularity. Notably, this is true also when the reaction term is not present, i.e. $a=0$.
\medskip

The next question that we want to address is related to the ``speed of
invasion'' of the fundamental solution. Namely, if~$u$ is as given by Theorem~\ref{EXUN},
one can compute the values of~$u$ on invading spheres depending on time
and determine the asymptotics of these values for large time.
Concretely,
one can consider an increasing function~$\Theta(t)$ and
look at the fundamental solution~$u$
at~$|x|=\Theta(t)$. Roughly speaking, one expects that if~$\Theta$ is ``small enough''
(that is, one moves along ``sufficiently slow'' spheres) then the values of~$u$
will be largely affected by the ``infinite'' initial datum at the origin,
and hence one expects that, in this case,
\begin{equation}\label{TT1} \lim_{t\to+\infty} u( \Theta(t) e,t)=+\infty,\end{equation}
for a given direction~$e\in {\mathbb{S}}^{N-1}$.
Viceversa, if~$\Theta$ is ``large enough''
(that is, one moves along ``sufficiently fast'' spheres) then the values of~$u$
will rapidly drift away from the initial datum at the origin
and quickly approach the datum at infinity: in this case, one expects that
\begin{equation}\label{TT2} \lim_{t\to+\infty} u( \Theta(t) e,t)=0.\end{equation}
Understanding the ``optimal'' velocity~$\Theta$
which produces the switch between the diverging behavior in~\eqref{TT1}
and the vanishing one in~\eqref{TT2} is clearly an important question both
in view of the development of the theory and for potential applications.
Since the general situation of equation~\eqref{eq:lin_reac_diff}
is likely to be extremely rich in complications, we focus here on the special
case of spatial dimension equal to~$1$ and fractional parameter~$\alpha$
bigger or equal than~$\frac12$. In this case, we establish that
the divergent behavior in~\eqref{TT1} occurs for all functions~$\Theta$
that evolve slower than the square root of time. The precise result that we obtain is the following:

\begin{theorem}\label{prop:N=1}
Let~$N=1$, $\alpha\in\left[\frac12,1\right)$ and~$a>0$.
Let~$u$ be as in Theorem~\ref{EXUN}.
Let also~$\beta\in\left(0,\frac12
\right)$ and~$c>0$.
Then,
\[
u(ct^\beta,t) \ge c\alpha t^{\beta - 2\alpha}
E_\alpha\left(t^{\alpha}\right)\left[1 -
o(1)\right].
\]
In particular,
\begin{equation}\label{DIVE} u(ct^\beta,t)
\to+\infty\qquad\text{ as $t\to+\infty$}.\end{equation}
\end{theorem}

The proof of Theorem~\ref{prop:N=1} is delicate and
will be based on a suitable analysis
exploiting some fine properties of the Mittag-Leffler function
and a number of ad-hoc simplifications to take care of some wildly oscillatory behaviors
of the terms describing the fundamental solution in terms of space-time power series.\medskip

Theorem~\ref{prop:N=1} is of clear interest in itself,
since it reveals an important physical feature of the ``subdiffusion''
provided by the time-fractional operator: namely, if the memory effect of the fractional
derivative is expected to ``slow down'' the invasion with respect to the classical case, in which the diffusion occurs with a linear speed, our result establishes that such an invasion still occurs,
in a power-like time as close as we wish to the square root function. Furthermore,
Theorem~\ref{prop:N=1} discloses a number
of very interesting research directions. A
few natural ones  are concerned with the  ranges of $\alpha$ and $\beta$
and with the spatial dimensions $N\geq 2$; we describe them as follows:

\begin{open} {\rm It would be desirable to find the {\em``optimal speed''}
distinguishing the diverging from the vanishing behavior in Theorem~\ref{prop:N=1},
and in general to {\em determine the asymptotics in the case~$\beta\ge\frac12$}.
With respect to this, we mention that the estimate found
in Theorem~\ref{prop:N=1} is valid somehow independently on~$\alpha$,
and in a way which treats the parameters~$\alpha$ and~$\beta$
in an essentially uncorrelated way. Nevertheless,
the classical case corresponding to~$\alpha=1$ would correspond to a linear
velocity of invasion (that is, $\beta=1$), and Theorem~\ref{prop:N=1} does
not capture this limit feature. In this sense,
it would be desirable to investigate whether it is possible to obtain
results such as in Theorem~\ref{prop:N=1} with parameters~$\alpha$
and~$\beta$ in a clear and explicit correlation that, on the one hand,
underlines the memory effect of the subdiffusion process, and, on the other hand,
{\em recovers the classical linear speed of invasion} in the limit as~$\alpha\nearrow1$.
In general,
when~$\beta\ge\frac12$, the asymptotic
and cancellation properties of the fundamental solution
are likely to be different than the ones discussed in this paper and a new approach
has probably to be taken into account.
}\end{open}

\begin{open}{\rm It would be interesting to obtain a result in the spirit
of Theorem~\ref{prop:N=1} for
the {\em highly nonlocal regime~$\alpha\in\left(0,\frac12\right)$}.
In this regime, the asymptotics of the Mittag-Leffler function
are different than those considered in this paper and thus new ingredients
have to be taken into consideration.
}\end{open}

\begin{open}{\rm It would be desirable to understand the asymptotics
of the fundamental solution in {\em every spatial dimension~$N\ge2$}.
In this case, additional terms related to the rotational behavior of the Fourier Transform
have to be taken into account. It is possible that in this case
the additional factors have to be understood in terms of special functions,
such as the the zeroth order Bessel function~$J_0$, and analyzed
in view of analytic tools such as the
Hankel Transform. A number of technical difficulties are expected to surface
in this case, since several useful identities involving the Mittag-Leffler function
and its derivatives would be affected by a possible sign change in their arguments,
leading to different types of cancellations.
}\end{open}

Theorem~\ref{prop:N=1} is also related
to a number of results in the recent literature dealing with the asymptotics
of nonlocal heat equation and of possibly nonlinear fractional equations.
With respect to this point, we observe that the nonlinear setting in the nonlocal
case happens to be better understood in the case of space-fractional,
rather than time-fractional, equations: for instance, in~\cite{MR2990897},
\cite{MR3057187} and~\cite{MR3698163}
a very detailed long time asymptotics is given for nonlinear
space-fractional equations. Conversely, the case of time-fractional equations
seems to be more difficult to consider, since the memory effect
given by the Caputo derivative does not often permit a direct use of barriers
and maximum principles in nonlinear scenarios. For time-fractional
equations with
large time decay, several results have been recently obtained in~\cite{MR2047909},
\cite{MR2125407},
\cite{MR2538276}, \cite{MR2718161}, \cite{MR2729240}, \cite{MR2872116},
\cite{MR3081208},
\cite{MR3156624}, \cite{MR3296607},
\cite{MR3563229}, \cite{MR3624548}, \cite{MR3631303}, \cite{2017arXiv170708278D},
\cite{MR3912710} and~\cite{MR3917402}.
Differently from the previous literature, in this paper we consider a
reaction term, namely, the term~$au$ on the right hand side of~\eqref{eq:lin_reac_diff}.
This term deeply affects the analysis of the problem,
since when~$a>0$ the arguments of the Mittag-Leffler functions
involved in the computations can become positive, thus exhibiting an
exponential behavior for large times.

Of course, this sign change is not only a merely technical occurrence,
but it is indeed responsible of the diverging structure described
in formula~\eqref{DIVE}. Indeed, while in the existing literature
the time-fractional heat equations were studied without a reaction term,
obtaining decay results for large times, Theorem~\ref{prop:N=1}
here provides the first result of divergence for large times,
in view of the positive reaction term when~$a>0$.\medskip

The rest of this article is organized as follows.
Section~\ref{DIE} is devoted to the proof of Theorem~\ref{EXUN}.
First, we will focus on the uniqueness result in Theorem~\ref{EXUN},
which will be the consequence of a series of general integral estimates.
For this, we will combine suitable methodologies of ordinary differential equations,
integral equations, and distribution theory.
Then, the existence result of Theorem~\ref{EXUN} will be proved
by explicitly checking that the setting in~\eqref{MIT} provides a solution of this problem.
This part is also not straightforward, since one has to check that
the definition of~\eqref{MIT} is compatible with the notion of solution
and that some of the ``formal'' computations that one would like to perform
are indeed well-justified in a functional analysis framework.

In Section~\ref{TRE} we recall some useful properties of the
Mittag-Leffler function. They will be exploited in Section~\ref{QUAT},
where we prove Theorem~\ref{prop:N=1}. For this, we will need to
carefully analyze the cancellations occurring in the Fourier Transform
representation of the fundamental solution. These cancellations
are dictated by the oscillatory kernel of the Fourier Transform
and one of the key ingredients of our analysis will be to discover
an alternate cancellation structure similar to that of the ``Leibniz Criterion''
for numerical series.

\section{The fundamental solution of the reaction-diffusion equation}\label{DIE}

In this section,
we discuss the existence and uniqueness
of the fundamental solution of~\eqref{eq:lin_reac_diff}.

The main idea is that such a fundamental solution can be written
explicitly in terms of the Mittag-Leffler function (recall~\eqref{MIT}). This explicit representation
will be valid in any dimension. Moreover, the fundamental solution
turns out to be unique in the sense of distributions.
The technical computations needed to check uniqueness
are contained in Subsection~\ref{UNIQ}, while
the existence and explicit representation of the solution are discussed
in Subsection~\ref{EXIS}.
In our setting, Theorem~\ref{EXUN} will follow from the subsequent
results in Corollary~\ref{7:92} and Lemma~\ref{lem:fund_sol}.

\subsection{The fundamental solution of the reaction-diffusion equation:
uniqueness of the solution}\label{UNIQ}

We discuss now the uniqueness result claimed in Theorem~\ref{EXUN}.
To this end, we establish first a pivotal result, which will
then be applied to the difference of two possible solutions.

\begin{lemma}\label{wun}
Let~$w\in C^\infty(\R^n)$.
Let~$T>0$ and
\begin{equation}\label{v as}
v\in L^\infty_\alpha\big([0,T], H^{-m}(\R^N)\big)\end{equation}
be such that, for all~$\varphi\in C^\infty_0(\R^N)$ and~$t\in[0,T]$,
$$ \langle S(t),\varphi\rangle=
\frac{1}{\Gamma(\alpha)}\int_0^t \frac{\left\langle S(\tau),
w\varphi
\right\rangle}{(t-\tau)^{1-\alpha}}\,d\tau ,$$
where~$S$ is the Fourier Transform of~$v$.
Then~$v(t)=0$ for all~$t\in[0,T]$.
\end{lemma}

\begin{proof}
We fix~$R\ge1$, and
we take~$\phi\in C^\infty_0( B_R)$.
We see that (up to normalizing constants
that we omit for the sake of simplicity)
\begin{eqnarray*} &&\|{\mathcal{F}}(w\phi)\|_{H^{m}(\R^N)}
=\sum_{|\beta|\le m} \big\| D^\beta\big({\mathcal{F}}(w\phi)\big)\big\|_{L^2(\R^N)}
=\sum_{|\beta|\le m} \big\| {\mathcal{F}}\big(\xi^\beta w\phi\big)\big\|_{L^2(\R^N)}
\\&&\qquad
=\sum_{|\beta|\le m} \big\| \xi^\beta w\phi\big\|_{L^2(\R^N)}\le
\|w\|_{L^\infty(B_R)}\sum_{|\beta|\le m} \big\| \xi^\beta \phi\big\|_{L^2(\R^N)}\\&&\qquad=
\|w\|_{L^\infty(B_R)}\sum_{|\beta|\le m} \big\| {\mathcal{F}}
\big(\xi^\beta \phi\big)\big\|_{L^2(\R^N)}
=
\|w\|_{L^\infty(B_R)}\sum_{|\beta|\le m}\| D^\beta \big({\mathcal{F}}(\phi)\big)\|_{L^2(\R^N)}\\ &&\qquad=
\|w\|_{L^\infty(B_R)}\,\|{\mathcal{F}}(\phi)\|_{H^{m}(\R^N)}.
\end{eqnarray*}

We define
$$ \sigma(t):= \sup_{\tau\in[0,t]} \;
\sup_{\psi\in C^\infty_0(B_R)\setminus\{0\}
}\frac{t^\alpha
\big| \langle S(t),\psi\rangle\big|
}{\|\hat\psi \|_{H^{m}(\R^N)}},$$
where~$\hat\psi={\mathcal{F}}(\psi)$ denotes the Fourier Transform of~$\psi$.

We point out that
$$ \frac{
\big| \langle S(t),\psi\rangle\big|
}{\|\hat\psi \|_{H^{m}(\R^N)}}=\frac{
\big| \langle v(t),\hat\psi\rangle\big|
}{\|\hat\psi \|_{H^{m}(\R^N)}}$$
and thus,
by~\eqref{v as} and~\eqref{Dedf2},
we see that~$\sigma(t)\le C$, for some~$C>0$,
for all~$t\in[0,T]$ (in particular,
$\sigma(t)$ is finite). Moreover, for all~$t\in[0,T]$ and all~$\psi\in C^\infty_0(B_R)$,
$$ \big| \langle S(t),\psi\rangle\big|\le\frac{ \sigma(t)\,\|\hat\psi \|_{H^{m}(\R^N)}}{t^\alpha}.$$
In this way, for all~$T_0\ge0$ and all~$t\in[0,T_0]$, we have that
\begin{eqnarray*}
\big|\langle S(t),\phi\rangle\big|&\le&
\frac{1}{\Gamma(\alpha)}\int_0^t \frac{\left|\left\langle S(\tau),
w\phi
\right\rangle\right|}{\tau^\alpha (t-\tau)^{1-\alpha}}\,d\tau\\
&\le&\frac{1}{\Gamma(\alpha)}\int_0^t \frac{\sigma(\tau)\,
\left\|\widehat{w\phi}\right\|_{H^{m}(\R^N)}}{\tau^\alpha (t-\tau)^{1-\alpha} }\,d\tau\\&\le&
\frac{\|w\|_{L^\infty(B_R)}\,\|\hat\phi\|_{H^{m}(\R^N)}}{\Gamma(\alpha)}\int_0^t \frac{
\sigma(\tau)}{\tau^\alpha (t-\tau)^{1-\alpha} }\,d\tau\\&\le&
\frac{\|w\|_{L^\infty(B_R)}\,\|\hat\phi\|_{H^{m}(\R^N)}\;\sigma(T_0)\;\Xi(T_0)}{\Gamma(\alpha)\;\alpha},
\end{eqnarray*}
for all~$t\in[0,T_0]$,
where
$$ \Xi(t):=\int_0^t \frac{
d\tau}{\tau^\alpha (t-\tau)^{1-\alpha} }=
\frac{4^{\alpha-1}\sqrt{\pi} \;  \Gamma(1 - \alpha) \;t^{2 - 2 \alpha}}{\Gamma(3/2 - \alpha)}.$$
As a consequence, we obtain that
$$ \sigma(T_0)\le
\frac{\|w\|_{L^\infty(B_R)}\;\sigma(T_0)\;T_0^\alpha\;\Xi(T_0)}{\Gamma(\alpha)\;\alpha}.$$
In particular, if~$ T_0$ is sufficiently small,
it follows that~$\sigma(T_0)\le\sigma(T_0)/2$, and therefore~$\sigma(T_0)=0$.

Let now
$$ T_{*}:=\sup\Big\{t\in[0,T]
{\mbox{ s.t. }} \sigma(t)=0  \Big\}.$$
We have just proved that~$T_{*}>0$, and we now claim that
\begin{equation}\label{89TAasggado78}
T_{*}=T.
\end{equation}
For this, suppose by contradiction that~$T_{*}\in(0,T)$.
Then, for all~$\epsilon\in[0,T-T_{*}]$, and all~$t\in[0,T_{*}+\epsilon]$,
we set~$\tilde{t}:=\max\{t,T_{*}\}$ and we have that
\begin{eqnarray*}
\big|\langle S(t),\phi\rangle\big|&\le&
\frac{1}{\Gamma(\alpha)}\int_{T_{*}}^{\tilde{t}} \frac{\left|\left\langle S(\tau),
w\phi
\right\rangle\right|}{(\tilde{t}-\tau)^{1-\alpha}}\,d\tau\\
&\le&
\frac{1}{\Gamma(\alpha)}\int_{T_{*}}^{\tilde{t}} \frac{\sigma(\tau)\,\|
{\mathcal{F}}(
w\phi)\|_{H^{m}(\R^N)}}{\tau^\alpha(\tilde{t}-\tau)^{1-\alpha}}\,d\tau
\\ &\le& \frac{
\sigma(T_{*}+\epsilon)\,\|w\|_{L^\infty(B_R)}\,\|\hat\phi\|_{H^{m}(\R^N)}
}{\Gamma(\alpha)}\int_{T_{*}}^{\tilde{t}} \frac{d\tau
}{\tau^\alpha(\tilde{t}-\tau)^{1-\alpha}}\\
&=& \frac{
\sigma(T_{*}+\epsilon)\,\|w\|_{L^\infty(B_R)}\,\|\hat\phi\|_{H^{m}(\R^N)}
\;\big(
\Xi(\tilde{t})-\Xi(T_{*})\big)
}{\Gamma(\alpha)\;\alpha}
.\end{eqnarray*}
As a consequence,
$$ \sigma(T_{*}+\epsilon)\le\frac{\sigma(T_{*}+\epsilon)\,\|w\|_{L^\infty(B_R)}\,
(T_{*}+\epsilon)^\alpha \;\big(
\Xi(T_{*}+\epsilon)-\Xi(T_{*})\big)}{\Gamma(\alpha)\;\alpha}.$$
In particular, for a suitably small~$\epsilon>0$,
$$ \sigma (T_{*}+\epsilon)\le\frac{ \sigma (T_{*}+\epsilon) }{2},$$
which gives that~$\sigma (T_{*}+\epsilon)=0$.
This contradicts the maximality of~$T_{*}$ and so it proves~\eqref{89TAasggado78}.

In particular, we obtain that~$\langle S(t),\phi\rangle=0$
for all~$t\in[0,T]$, for all~$\phi\in C^\infty_0(B_R)$, and for all~$R>0$. Hence,
we find that~$\langle S(t),\phi\rangle=0$
for all~$t\in[0,T]$ and~$\phi\in C^\infty_0(
\R^n)$. Since~$C^\infty_0(
\R^n)$ is dense in~$H^{m}(\R^N)$,
we conclude that~$\langle S(t),\phi\rangle=0$
for all~$t\in[0,T]$ and~$\phi\in H^{m}(\R^N)$.
Hence~$v(t)=\hat S(t)=0$, as desired.
\end{proof}

{F}rom this, we can establish the uniqueness claim in Theorem~\ref{EXUN}:

\begin{corollary}\label{7:92}
There exists at most one solution of~\eqref{eq:lin_reac_diff}.
\end{corollary}

\begin{proof} Suppose that~$u_1$ and~$u_2$ are both solutions
of~\eqref{eq:lin_reac_diff}, and let~$v:=u_1-u_2$. Since~$v(t)\in H^{-m}(\R^N)$,
we can consider its Fourier Transform in the distributional sense,
that we denote by~$S(t)$. In view of~\eqref{DEdebo}, we know that
$$ \langle v(t),\psi\rangle=\frac{1}{\Gamma(\alpha)}\int_0^t \frac{
d\langle v(\tau),\Delta\psi \rangle+a\langle v(\tau),\psi\rangle
}{(t-\tau)^{1-\alpha}}\,d\tau,$$
for all~$\psi\in H^{m}(\R^N)$.

Hence, given any~$\varphi\in C^\infty_0(\R^N)$, we can apply the previous identity
to the Fourier Transform of~$\varphi$,
that we denote by~$\psi:=\hat\varphi$, and find that
\begin{equation} \label{th:u2}
\begin{split}
\langle S(t),\varphi\rangle\,&=\langle v(t),\psi\rangle\\
&=\frac{1}{\Gamma(\alpha)}\int_0^t \frac{
d\langle v(\tau),\Delta\hat\varphi \rangle+a\langle v(\tau),\hat\varphi\rangle
}{(t-\tau)^{1-\alpha}}\,d\tau,\end{split}\end{equation}
for all~$\varphi\in C^\infty_0(\R^N)$.

We also set~$\eta(x):=|x|^2\varphi(x)\in C^\infty_0(\R^N)$. We observe that
\begin{eqnarray*} &&\Delta\hat\varphi (\xi)=\sum_{j=1}^N \partial_{\xi_j}^2
\int_{\R^N} \varphi(x)\,e^{-2\pi{\rm i} x\cdot\xi}\,dx=
-4\pi^2\sum_{j=1}^N
\int_{\R^N} x_j^2\varphi(x)\,e^{-2\pi{\rm i} x\cdot\xi}\,dx\\
&&\qquad=-4\pi^2
\int_{\R^N} \eta(x)\,e^{-2\pi{\rm i} x\cdot\xi}\,dx=-4\pi^2\hat\eta(\xi).
\end{eqnarray*}
Notice also that
$$ -4\pi^2 d\eta(x)+a\varphi(x)=
(-4\pi^2 d|x|^2+a)\varphi(x)=w(x)\varphi(x),$$
where
$$ w(x):=-4\pi^2 d|x|^2+a.$$
Consequently,
\begin{eqnarray*}&&
d\langle v(\tau),\Delta\hat\varphi \rangle+a\langle v(\tau),\hat\varphi\rangle=
-4\pi^2 d\langle v(\tau),\hat\eta \rangle+a\langle v(\tau),\hat\varphi\rangle\\
&&\qquad=-4\pi^2 d\langle S(\tau),\eta \rangle+a\langle S(\tau),\varphi\rangle=
\langle S(\tau),w\varphi \rangle.
\end{eqnarray*}
We plug this information into~\eqref{th:u2} and we find that
\begin{equation*}\langle S(t),\varphi\rangle=\frac{1}{\Gamma(\alpha)}\int_0^t \frac{
\langle S(\tau),w\varphi \rangle
}{(t-\tau)^{1-\alpha}}\,d\tau,\end{equation*}
for all~$\varphi\in C^\infty_0(\R^N)$.

Then, in light of Lemma~\ref{wun}, we conclude that~$v(t)=0$, and thus~$u_1(t)=u_2(t)$,
for all~$t\in[0,T]$.
\end{proof}

\subsection{The fundamental solution of the reaction-diffusion equation:
existence of the solution}\label{EXIS}

Here we focus on the existence statement in Theorem~\ref{EXUN},
which will be obtained via the representation formula in~\eqref{MIT}
that exploits the Mittag-Leffler function~$E_\alpha$.
In our framework,
the crucial role played by~$E_{\alpha}$ consists in the fact that
\begin{equation}\label{DE} \capu \big( E_{\alpha}(\lambda t^{\alpha})\big)=\lambda
E_{\alpha}(\lambda t^{\alpha}),\end{equation}
for all~$\lambda\in \R$,
see e.g.~\cite[Theorem 4.3]{di}.

Moreover, it is convenient to extend the setting in~\eqref{eq:defML} by also defining
$$ E_{\alpha,\alpha  }(r):=\sum_{k=0}^{\infty} \frac{r^k}{\Gamma(\alpha+k\alpha)}.
$$
For later use, we recall some well-known properties of
the Mittag-Leffler functions in the following lemma.
\begin{lemma}\label{lem:ML_basics}
Let $0<\alpha<1$. Then
\begin{align}
E_\alpha'(r)&= \dfrac{1}{\alpha} E_{\alpha,\alpha  }(r) &&\text{for every $r\in\R$,}&
\label{eq:der_alfalfa}\\
E_\alpha(r)&>0,\ E_{\alpha,\alpha  }(r) >0 &&\text{for every $r\in\R$,}
\label{eq:ML_complMonoton}\\
E_\alpha(r) &= \frac{1}{\alpha}\exp{r^{1/\alpha}} + O\left(\frac{1}{r}\right), &&r\to+\infty
\label{eq:ML_asynt+infty},\\
E_\alpha(r) &= -\frac{1}{\Gamma(1-\alpha)}\,\frac{1}{r} + O\left(\frac{1}{r^2}\right), &&r\to-\infty.
\label{eq:ML_asynt-infty}
\end{align}
\end{lemma}
\begin{proof}
All the properties are well-known: \eqref{eq:der_alfalfa} is immediate, recalling that
$\Gamma(1+(k+1)\alpha) = \alpha(k+1) \Gamma((k+1)\alpha)$; \eqref{eq:ML_complMonoton} is trivial
for $r\ge0$, while for $r<0$ it descends from \eqref{eq:der_alfalfa}, and from the fact that
$E_\alpha(-z)$ is completely monotonic for $z>0$ and $0\le\alpha\le1$ (see e.g. \cite{MR0066496},
eq. (6) on p. 207); finally, \eqref{eq:ML_asynt+infty} and \eqref{eq:ML_asynt-infty} are equations (10) and (7) on p. 208 and p. 207 in \cite{MR0066496}, respectively.
\end{proof}
\begin{remark}
Since $E_1(r)=e^r$, the claims in Lemma~\ref{lem:ML_basics}
hold true also when $\alpha=1$, except of course
\eqref{eq:ML_asynt-infty}.
\end{remark}


\begin{lemma}\label{lem:fund_sol}
Let~$T>0$.
Problem \eqref{eq:lin_reac_diff} is solved in~$t\in[0,T]$ by
\begin{equation}\label{eq:ualpha}
u(x,t) :={\mathcal{F}}^{-1} \Big(E_{\alpha}\big((a-4\pi^2
d|\xi|^{2})t^{\alpha}\big)\Big),
\end{equation}
being~${\mathcal{F}}$ the Fourier Transform.
\end{lemma}

\begin{proof}
First of all, we show that, if~$u$ is as in~\eqref{eq:ualpha},
for every~$t\in[0,T]$,
\begin{equation}\label{good}
u(t)\in H^{-m}(\R^N),\quad {\mbox{ and }}\quad \sup_{t\in[0,T]}\frac{\langle u(t),\phi\rangle}{
\|\phi\|_{H^{m}(\R^N)}}<+\infty,
\end{equation}
as long as
\begin{equation}\label{mgra}\N\ni
m>\frac{N-4}2.\end{equation}
To check this, we first observe that,
for every smooth and compactly supported function~$\phi$ and any~$\alpha\in\N^N$,
considering the Fourier Transform~$\hat\phi={\mathcal{F}}\phi$,
we have that
$$ \| D^\alpha\phi\|_{L^2(\R^N)}^2
=\| {\mathcal{F}} (D^\alpha\phi)\|_{L^2(\R^N)}^2
=\| \xi^\alpha\hat\phi\|_{L^2(\R^N)}^2.$$
Also, if~$\rho:=\frac{1}{\pi}\sqrt{\frac{a}{2d}}$ and~$\xi\in\R^N\setminus B_\rho$, we have that
$$ 2\pi^2 d|\xi|^2-a \ge0$$
and therefore
$$ \frac{1}{4\pi^2 d|\xi|^2-a}\le \frac{1}{2\pi^2 d|\xi|^2}.
$$
{F}rom this and~\eqref{eq:ML_asynt-infty}, we see that
\begin{equation}\label{09:348}
\begin{split}
&\left|\int_{ \R^N\setminus B_\rho }
E_{\alpha}\big((a-4\pi^2 d|\xi|^{2})t^{\alpha}\big)\,\hat\phi(\xi)\,d\xi\right|\le
C\,\int_{ \R^N\setminus B_\rho }
\frac{|\hat\phi(\xi)|}{(4\pi^2 d|\xi|^{2}-a)t^{\alpha}}\,d\xi\\
&\qquad\le
C\,\int_{ \R^N\setminus B_\rho }
\frac{|\hat\phi(\xi)|}{2\pi^2 dt^{\alpha}|\xi|^{2}}\,d\xi
=\frac{C}{t^\alpha}\,
\int_{ \R^N\setminus B_\rho }
\frac{ |\xi|^{{m}} |\hat\phi(\xi)| }{ |\xi|^{2+{{m}}}}\,d\xi
\\
&\qquad\le\frac{C}{t^\alpha}\,\sqrt{\int_{ \R^N\setminus B_\rho }|\xi|^{2m}|\hat\phi(\xi)|^2\,d\xi}
\,\sqrt{\int_{ \R^N\setminus B_\rho }
\frac{d\xi}{|\xi|^{4+2m}}\,d\xi}\\&\qquad \le
\frac{C}{t^\alpha}\,\sqrt{\int_{ \R^N }|\xi|^{2m}|\hat\phi(\xi)|^2\,d\xi}
\le\frac{C}{t^\alpha}\,\| D^m \phi\|_{L^2(\R^N)},
\end{split}
\end{equation}
for some~$C>0$ varying from line to line and where~\eqref{mgra} has been exploited.

On the other hand,
\begin{eqnarray*}&&
\left|\int_{B_\rho }
E_{\alpha}\big((a-4\pi^2 d|\xi|^{2})t^{\alpha}\big)\,\hat\phi(\xi)\,d\xi\right|\le
C\,\int_{B_\rho }
|\hat\phi(\xi)|\,d\xi\le
C\,\sqrt{\int_{B_\rho }
|\hat\phi(\xi)|^2\,d\xi}\\
&&\qquad\le C\,\|\hat\phi\|_{L^2(\R^N)}
=C\,\|\phi\|_{L^2(\R^N)}.
\end{eqnarray*}
This and~\eqref{09:348}, up to renaming~$C$, yield that
\begin{eqnarray*}
\frac{C\,\|\phi\|_{H^{m}(\R^N)}}{t^\alpha}&\ge&
\left|\int_{ \R^N }
E_{\alpha}\big((a-4\pi^2 d|\xi|^{2})t^{\alpha}\big)\,\hat\phi(\xi)\,d\xi\right|\\
&=&
\left|\int_{ \R^N }
{\mathcal{F}}^{-1}\Big(
E_{\alpha}\big((a-4\pi^2 d|\cdot|^{2})t^{\alpha}\big)\Big)\,\phi(x)\,dx\right|\\
&=& \big| \langle u(t),\phi\rangle\big|.
\end{eqnarray*}
This proves~\eqref{good}.

{F}rom this, we know that~$u\in L^\infty_\alpha\big([0,T], H^{-m}(\R^N)\big)$.
Hence, it remains to show that~\eqref{DEdebo} is satisfied.
To this end, we recall~\eqref{VOLTERRA} and~\eqref{DE},
and we write that
$$
E_\alpha(\lambda t^\alpha)=1+
\frac{\lambda}{\Gamma(\alpha)}\int_0^t \frac{E_\alpha(\lambda\tau^\alpha)}{
(t-\tau)^{1-\alpha}}\,d\tau .
$$
Using this formula with~$\lambda:=a-4\pi^2d|\xi|^{2}$, and recalling~\eqref{eq:ualpha},
we obtain that
$$ \hat u(\xi,t)=
E_\alpha\big((a-4\pi^2d|\xi|^{2}) t^\alpha\big)=1+
\frac{a-4\pi^2d|\xi|^{2}}{\Gamma(\alpha)}\int_0^t \frac{E_\alpha\big((a-4\pi^2d|\xi|^{2})\tau^\alpha\big)}{
(t-\tau)^{1-\alpha}}\,d\tau .
$$
As a consequence, for every~$\varphi\in C^\infty_0(\R^N)$,
\begin{eqnarray*}
\langle u(t),\varphi\rangle-\varphi(0)&=&
\langle \hat u(t)-1,\hat\varphi\rangle\\&=&
\int_{\R^N} \left[
\frac{(a-4\pi^2d|\xi|^{2})\hat\varphi(\xi)}{\Gamma(\alpha)}\int_0^t \frac{E_\alpha\big((a-4\pi^2d|\xi|^{2})\tau^\alpha\big)}{
(t-\tau)^{1-\alpha}}\,d\tau \right]\,d\xi
\\&=&
\int_{\R^N} \left[
\frac{a\hat\varphi(\xi)+d{\mathcal{F}}(\Delta\varphi)(\xi)}{\Gamma(\alpha)}\int_0^t \frac{E_\alpha\big((a-4\pi^2d|\xi|^{2})\tau^\alpha\big)}{
(t-\tau)^{1-\alpha}}\,d\tau \right]\,d\xi\\
\\&=& \frac{1}{\Gamma(\alpha)}\int_0^t
 \left[ \int_{\R^N}
\frac{\big(a\hat\varphi(\xi)+d{\mathcal{F}}(\Delta\varphi)(\xi)\big)\;
E_\alpha\big((a-4\pi^2d|\xi|^{2})\tau^\alpha\big)}{
(t-\tau)^{1-\alpha}}\,d\xi \right]\,d\tau\\
&=&\frac{1}{\Gamma(\alpha)}\int_0^t\frac{
\langle a\hat\varphi+d{\mathcal{F}}(\Delta\varphi),\,\hat u(\tau)\rangle
}{(t-\tau)^{1-\alpha}}\,d\tau\\
&=&\frac{1}{\Gamma(\alpha)}\int_0^t\frac{
\langle a \varphi+d \Delta\varphi,\, u(\tau)\rangle
}{(t-\tau)^{1-\alpha}}\,d\tau,
\end{eqnarray*}
and this establishes~\eqref{DEdebo}.
\end{proof}

%
\begin{remark}
In case $\alpha=1$, then \eqref{eq:lin_reac_diff} reduces to the standard heat equation, and
recalling that $E_1(r)=e^r$ Lemma \ref{lem:fund_sol} provides the well-known expression of
the solution:
\[
u(x,t) = \frac{1}{(4\pi dt )^{N/2}}\exp{\left( at - \frac{|x|^2}{4dt}\right)}.
\]
\end{remark}

\section{Some estimates about Mittag-Leffler functions}\label{TRE}

This section presents some useful facts
on the Mittag-Leffler function and on similar special functions.
These properties will be utilized in the next section to prove Theorem~\ref{prop:N=1}.
For other interesting properties
of the Mittag-Leffler function see e.g.~\cite{MR1906764}, \cite{MR3253257},
and the references therein.

\begin{lemma}\label{lem:MLda sopra}
Let $\alpha\ge 1/2$ and $r\ge0$. Then
\[
E_\alpha(r) \le \alpha E'_\alpha(r) + 1 - \frac{1}{\Gamma(\alpha)}
 + \left(\frac{1}{\Gamma(1+\alpha)} - \frac{1}{\Gamma(2\alpha)} \right)r.
\]
\end{lemma}
\begin{proof}
It follows by \eqref{eq:defML}, \eqref{eq:der_alfalfa}, and by the fact that
\[
s\ge\frac32 \implies \Gamma'(s)>0,
\qquad
\text{i.e. }
k\ge 2 \implies \frac{1}{\Gamma(1+k\alpha)}
\le  \frac{1}{\Gamma(\alpha+k\alpha)}
\]
(this is well known, see e.g. \cite{DECO_N1935}). Indeed,
\[
E_\alpha(r)  - 1 - \frac{r}{\Gamma(1+\alpha)}
=    \sum_{k=2}^{\infty} \frac{r^k}{\Gamma(1+k\alpha)}
\le  \sum_{k=2}^{\infty} \frac{r^k}{\Gamma(\alpha+k\alpha)}
=    E_{\alpha,\alpha  }(r) - \frac{1}{\Gamma(\alpha)}  - \frac{r}{\Gamma(2\alpha)}.
\qedhere
\]
\end{proof}
\begin{lemma}\label{lem:MLda sotto}
Let $\alpha\ge 1/2$ and $r>0$. Then
\[
E_\alpha(r) \ge
\frac{\alpha}{ r} E'_\alpha(r) - \frac{1}{\Gamma(\alpha)r} + 1 - \frac{1}{\Gamma(2\alpha)}.
\]
\end{lemma}
\begin{proof}
We have
\[
rE_{\alpha}(r)=\sum_{k=0}^{\infty} \frac{r^{k+1}}{\Gamma(1+k\alpha)}
=\sum_{k=1}^{\infty} \frac{r^{k}}{\Gamma(1-\alpha+k\alpha)}.
\]
Since $\alpha\ge1/2$ we have $1-\alpha \le \alpha$. Using again the monotonicity of $\Gamma$
we infer
\[
k\ge 2 \implies \frac{1}{\Gamma(1-\alpha+k\alpha)}
\ge  \frac{1}{\Gamma(\alpha+k\alpha)},
\]
and
\[
rE_\alpha(r)  - r
=    \sum_{k=2}^{\infty} \frac{r^k}{\Gamma(1-\alpha+k\alpha)}
\ge  \sum_{k=2}^{\infty} \frac{r^k}{\Gamma(\alpha+k\alpha)}
=    E_{\alpha,\alpha  }(r) - \frac{1}{\Gamma(\alpha)}  - \frac{1}{\Gamma(2\alpha)}r.
\qedhere
\]
\end{proof}
\begin{remark}
Since
\[
E_1(r)=e^r,\qquad E_{1/2}(r) = \left(1+\frac{2}{\sqrt{\pi}}\int_{0}^{r}e^{-s^2}\,ds\right)e^{r^2},
\]
we obtain
\[
E'_1(r) = E_1(r),\qquad E'_{1/2}(r) = 2r E_{1/2}(r) + \dfrac{2}{\sqrt{\pi}}.
\]
In particular, up to the lower order terms, Lemma \ref{lem:MLda sopra} is optimal for $\alpha = 1$, while Lemma \ref{lem:MLda sotto} is optimal for $\alpha = 1/2$.
\end{remark}

\section{Estimates of the fundamental solution}\label{QUAT}

The goal of this section is to prove Theorem~\ref{prop:N=1}. To this end,
in the following we study the fundamental solution $u$ defined in Theorem~\ref{EXUN},
restricting to dimension $N=1$. {F}rom now on,
we take $a:=1$ and~$d:=1$,
the general case following by suitable change of variables in
the integral defining $u$.
In this way, formula~\eqref{MIT} becomes
\begin{equation}\label{la u}
\begin{split}
u(x,t) \,&=\int_{\R}E_{\alpha}\big((1-4\pi^2
|\xi|^{2})t^{\alpha}\big)\,\cos(2\pi x \xi)\,d\xi\\
&=2\int_0^{+\infty}E_{\alpha}\big((1-4\pi^2
|\xi|^{2})t^{\alpha}\big)\,\cos(2\pi x \xi)\,d\xi\\
&=\frac1\pi \int_0^{+\infty}E_{\alpha}\big((1-|\xi|^{2})t^{\alpha}\big)\,\cos(x \xi)\,d\xi.
\end{split}\end{equation}
We will now prove Theorem~\ref{prop:N=1} through a sequence of lemmas.
The first of this lemmas splits the integral into~\eqref{la u}
into an infinite sum. This will be convenient in order to detect the alternate cancellations
arising from the oscillatory kernel of the Fourier Transform.
\begin{lemma}
We have that
\begin{equation}\label{la u2}\begin{split}
u(x,t) &=\frac1\pi \sum_{k=0}^\infty (-1)^k a_k(x,t),\\
{\mbox{with }}\quad
a_0 := a_0(x,t) &:= \int_{0}^{\frac{\pi}{2x}}
E_{\alpha}((1-\xi^{2})t^{\alpha})\cos{(x\xi)}\,d\xi \\
a_k:=a_k(x,t) &:= (-1)^k \int_{\frac{\pi}{2x}(2k-1)}^{\frac{\pi}{2x}(2k+1)}
E_{\alpha}((1-\xi^{2})t^{\alpha})\cos{(x\xi)}\,d\xi ,\qquad {\mbox{ for }}k\ge1.
\end{split}\end{equation}
Moreover, for every  $t>0$, $x\ge0$,  $k\in\N$, we have that
\begin{equation}\label{ak mag 0} a_{k}(x,t) > 0,\end{equation}
and, furthermore,
\begin{equation}\label{ak to 0}
a_k(x,t)\to0\;\text{ as }\;k\to+\infty  .
\end{equation}
In addition,
\begin{equation}\label{ak-1}
a_0>\frac{a_1}2,
\end{equation}
and, if $k\ge 1$,
\begin{equation}\label{ak-2}
a_{k+1}(x,t) < a_{k}(x,t).
\end{equation}
\end{lemma}
\begin{proof}
The claim in~\eqref{la u2} plainly follows from~\eqref{la u}.

Then,
since the integral defining $u$ converges, then also the series in~\eqref{la u2}
does, and consequently
$a_k\to0$ pointwise as $k\to+\infty$, which establishes~\eqref{ak to 0}.

Moreover, the positivity of  $a_k$ claimed in~\eqref{ak mag 0}
easily follows from \eqref{eq:ML_complMonoton}
and the fact that
\[
\frac{\pi}{2x}(2k-1)<\xi<\frac{\pi}{2x}(2k+1)
\qquad\implies\qquad
(-1)^k \cos{(x\xi)} >0.
\]
Finally, if $k\ge1$, we perform the change of variable $ \xi = \eta + \pi/x$ in $a_{k+1}$.
In this way, we infer that
\[
\begin{split}
a_{k+1}(x,t) &= (-1)^{k+1} 2\int_{\frac{\pi}{2x}(2k-1)}^{\frac{\pi}{2x}(2k+1)}
E_{\alpha}\left(\left(1-\left(\eta+\frac{\pi}{x}\right)^{2}\right)t^{\alpha}\right)
\cos{\left(x\eta + \pi\right)}\,d\eta\\
&= (-1)^{k} 2\int_{\frac{\pi}{2x}(2k-1)}^{\frac{\pi}{2x}(2k+1)}
E_{\alpha}\left(\left(1-\left(\eta+\frac{\pi}{x}\right)^{2}\right)t^{\alpha}\right)
\cos{\left(x\eta \right)}\,d\eta.
\end{split}
\]
Hence, by the strict monotonicity of $E_\alpha$,
see \eqref{eq:der_alfalfa} and~\eqref{eq:ML_complMonoton}, we conclude that
\[
a_{k+1}(x,t) < (-1)^{k} 2\int_{\frac{\pi}{2x}(2k-1)}^{\frac{\pi}{2x}(2k+1)} E_{\alpha}\left(\left(1-\eta^{2}\right)t^{\alpha}\right)
\cos{\left(x\eta \right)}\,d\eta = a_k(x,t),
\]
which concludes the proof of~\eqref{ak-2}.

Similarly, one sees that~\eqref{ak-1} holds true, since the
integral defining $a_{0}$ lies in $(0,\pi/2x)$.\end{proof}

Next result states the quantities~$a_k$ in an analytically convenient way.

%
\begin{lemma}\label{lem:change_rho}
Let $a_k$ be as in~\eqref{la u2}. Then,
\[
\begin{split}
a_0(x,t) &= \int_{1-\left(\frac{\pi}{2x}\right)^2
}^{1} E_{\alpha}(t^{\alpha}\rho)\frac{
\cos{\left(x\sqrt{1-\rho}\right)}}{\sqrt{1-\rho}}
\,d\rho \\
a_k(x,t) &= (-1)^k \int_{1-\left(\frac{\pi}{2x}\right)^2(2k+1)^2
}^{1-\left(\frac{\pi}{2x}\right)^2(2k-1)^2} E_{\alpha}(t^{\alpha}\rho)\frac{
\cos{\left(x\sqrt{1-\rho}\right)}}{\sqrt{1-\rho}}
\,d\rho ,\qquad k\ge1.
\end{split}
\]
\end{lemma}
\begin{proof} The desired claim follows
by a direct computation, using the (monotone decreasing) change of variables
\[
\rho := 1-\xi^{2} \in (-\infty,1],\qquad
\xi = \sqrt{1-\rho},\qquad
d\xi = -\frac{d\rho}{2\sqrt{1-\rho}}.\qedhere
\]
\end{proof}

The next result analyzes the ``first coefficient''~$a_0$, which
in the end will turn out to be the dominant one to understand the long time behavior
of the fundamental solution.

\begin{lemma}\label{lem:azero_dasotto}
If $1/2\le\alpha<1$, $t>0$, and
\[
0 < \ell < \frac{\pi}{2} < x,
\]
then
\[
a_{0}(x,t) \ge \frac{\alpha \cos\ell}{\ell} \,
\frac{x}{ t^{2\alpha}} \left[E_\alpha\left(t^{\alpha}\right) -
E_\alpha\left(t^{\alpha}\left(1-\frac{\ell^2}{x^2}\right)\right)+ c_{0}(x,t)  \right]
\]
where $c_{0}(x,t)$ is given by
\[
c_{0}(x,t) := \dfrac{t^{\alpha}}{\alpha\Gamma(\alpha)}\ln\left(1-\frac{\ell^{2}}{x^{2}}\right)
+\frac{\ell^{2}t^{2\alpha}}{\alpha x^{2}}  \left(1-\frac1{\Gamma(2\alpha)}\right)
.\]
\end{lemma}
\begin{proof}
By Lemma \ref{lem:change_rho}, for any $0<\ell(x)<\pi/2$ we have
\[
a_0(x,t) \ge \int_{1-\frac{\ell^2}{x^2}}^{1} E_{\alpha}(t^{\alpha}\rho)\frac{
\cos{\left(x\sqrt{1-\rho}\right)}}{\sqrt{1-\rho}}
\,d\rho \ge \frac{
x\cos\ell}{\ell}\int_{1-\frac{\ell^2}{x^2}}^{1} E_{\alpha}(t^{\alpha}\rho)
\,d\rho.
\]
The assumption on $x$ insures that $\rho>0$ in the integration interval. As a
consequence, using Lemma \ref{lem:MLda sotto}, we infer that
\[
\begin{split}
a_0(x,t) &\ge \frac{
x\cos\ell}{\ell}\int_{1-\frac{\ell^2}{x^2}}^{1} \left[
\frac{\alpha}{t^{\alpha}\rho} E'_\alpha(t^{\alpha}\rho)
- \frac{1}{\Gamma(\alpha)t^{\alpha}\rho} + 1 - \frac{1}{\Gamma(2\alpha)}\right]\,d\rho\\
&\ge \frac{\alpha x\cos\ell}{t^{2\alpha}\ell} \left[E_\alpha\left(t^{\alpha}\right) -
E_\alpha\left(t^{\alpha}\left(1-\frac{\ell^2}{x^2}\right)\right)
 \right]
 \\
 &
+\frac{  x\cos\ell}{ \ell}
\left[\dfrac{1}{t^{\alpha}\Gamma(\alpha)}\ln\left(1-\frac{\ell^{2}}{x^{2}}\right)
+\left(1-\frac1{\Gamma(2\alpha)}\right)\frac{\ell^{2}}{x^{2}}
\right]
\end{split}
\]
yielding the conclusion.
\end{proof}
We need now to compare the coefficient~$a_0$ with the other terms.
For this, we have the following result, estimating the next terms from above.

\begin{lemma}\label{lem:akappa_dasopra}
If $1/2\le\alpha<1$, $k\ge1$, $t>0$ and
\[
x > \frac{\pi}{2}(2k+1),
\]
then
\[
a_{k}(x,t)\le \frac{2\alpha}{(2k-1)\pi}\, \frac{x}{ t^{\alpha}} \left[E_\alpha\left(t^{\alpha}\left(1-\frac{\pi^2}{4x^2}(2k-1)^2\right)\right) -
E_\alpha\left(t^{\alpha}\left(1-\frac{\pi^2}{4x^2}(2k+1)^2\right)\right) + c_{k}(x,t) \right]
\]
where $c_{k}(x,t)$ is given by
\[
c_{k}(x,t):=
\frac{4k\pi }{x(2k-1) } \left(1-\frac1{\Gamma(\alpha)}\right)
+\frac{t^{\alpha}k\pi}{x(2k-1)}\left(\frac1{\Gamma(1+\alpha)}-\frac1{\Gamma(2\alpha)}\right)
\left(4-\left(\frac{\pi}{x}\right)^{2}(1+4k^{2}) \right).
\]
\end{lemma}
\begin{proof}
If $k\ge1$ then, using Lemma \ref{lem:MLda sopra},
\[
\begin{split}
a_k(x,t) =& (-1)^k \int_{1-\left(\frac{\pi}{2x}\right)^2(2k+1)^2
}^{1-\left(\frac{\pi}{2x}\right)^2(2k-1)^2} E_{\alpha}(t^{\alpha}\rho)\frac{
\cos{\left(x\sqrt{1-\rho}\right)}}{\sqrt{1-\rho}}\,d\rho\\
\le & \frac{2x}{(2k-1)\pi}\int_{1-\left(\frac{\pi}{2x}\right)^2(2k+1)^2
}^{1-\left(\frac{\pi}{2x}\right)^2(2k-1)^2} \alpha E'_\alpha(t^{\alpha}\rho) \,d\rho
\\
&+\frac{4k\pi }{x(2k-1) } \left(1-\frac1{\Gamma(\alpha)}\right)
\\
&+\frac{t^{\alpha}k\pi}{x(2k-1)}\left(\frac1{\Gamma(1+\alpha)}-\frac1{\Gamma(2\alpha)}\right)
\left(4-\left(\frac{\pi}{x}\right)^{2}(1+4k^{2}) \right)
,
\end{split}
\]
concluding the proof.
 \end{proof}

With the previous results, we are now in the position of completing the proof
of Theorem~\ref{prop:N=1} by detecting an infinite number of alternate cancellations.

%
%
\begin{proof}[Proof of Theorem \ref{prop:N=1}]
By the elementary Leibniz criterion for series with terms having alternate sign, we have that
\[
u(x,t) > a_0(x,t)-a_1(x,t).
\]
For $x>3\pi/2$ we can apply Lemma \ref{lem:azero_dasotto}, with $\cos\ell=\ell$, and Lemma \ref{lem:akappa_dasopra}, with $k=1$. We obtain
\[
\begin{split}
u(x,t) \ge
\frac{\alpha x}{t^{2\alpha}} \Bigg[E_\alpha\left(t^{\alpha}\right) &-
E_\alpha\left(t^{\alpha}\left(1-\frac{\ell^2}{x^2 }\right)\right)
 \\
& - t^\alpha E_\alpha\left(t^{\alpha}\left(1-\frac{A_-^2}{x^2}\right)\right) +
t^\alpha E_\alpha\left(t^{\alpha}\left(1-\frac{A_+^2}{x^2}\right)\right)
+ c_{0}(x,t)-c_{1}(x,t)
\Bigg],
\end{split}\]
where
\[
0 < \ell < A_-:=\frac{\pi}{2} < A_+:=\frac{3\pi}{2}.
\]
Let $m>0$ and $\beta>0$. Substituting $x= t^\beta/m$ we have,
for $t$ large,
\begin{multline*}
u\left(\frac{t^\beta}{m},t\right) \ge
\frac{\alpha}{m} t^{\beta - 2\alpha} E_\alpha\left(t^{\alpha}\right)\left[1 -
\frac{E_\alpha\left(t^{\alpha}(1-\ell^2 m^2 t^{-2\beta})\right)}{
E_\alpha\left(t^{\alpha}\right)}\right. \\
\left.-2 t^\alpha \frac{E_\alpha\left(t^{\alpha}(1-A_-^2 m^2 t^{-2\beta})\right)}{
E_\alpha\left(t^{\alpha}\right)} +2
t^\alpha \frac{E_\alpha\left(t^{\alpha}(1-A_+^2 m^2 t^{-2\beta})\right)}{
E_\alpha\left(t^{\alpha}\right)}
+ \frac{c_{0}\left(\frac{t^{\beta}}m,t\right)-c_{1}\left(\frac{t^{\beta}}m,t\right)}{E_\alpha\left(t^{\alpha}\right)}\right].
\end{multline*}
Since $\beta>0$, we infer that, as $t\to+\infty$, all the arguments of the
Mittag-Leffler  functions diverge to $+\infty$. Let us first study
the asymptotic behavior of $c_{0}\left(\frac{t^{\beta}}m,t\right)$ and
$c_{1}\left(\frac{t^{\beta}}m,t\right)$.
From \eqref{eq:ML_asynt+infty} we easily deduce that
$$
\frac{c_{0}\left(\frac{t^{\beta}}m,t\right)}{E_\alpha\left(t^{\alpha}
\right)}
=
\frac{\frac{t^{\alpha}}{\Gamma(\alpha)}\ln(1-\ell^{2}m^{2}t^{-2\beta})+\left(
1-\frac1{\Gamma(2\alpha)}\right) m^{2}\ell^{2}t^{2(\alpha-\beta)}}{\exp\left(t\right)+o(1)}
=o(1)$$
and
$$
\frac{c_{1}\left(\frac{t^{\beta}}m,t\right)}{
E_\alpha\left(t^{\alpha}\right)}=
\frac{4\pi mt^{-\beta} \left(1-\frac1{\Gamma(\alpha)}\right)
+t^{\alpha-\beta}m\pi \left(\frac1{\Gamma(1+\alpha)}-\frac1{\Gamma(2\alpha)}\right)
\left(4- 5(m\pi)^{2} t^{-2\beta}  \right)}{\exp\left(t\right)+o(1)}=o(1).
$$
Applying again \eqref{eq:ML_asynt+infty} we obtain
\[
\begin{split}
\frac{E_\alpha\left(t^{\alpha}(1-\ell^2 m^2 t^{-2\beta})\right)}{
E_\alpha\left(t^{\alpha}\right)} &= \frac{\exp\left(t(1-\ell^2 m^2 t^{-2\beta})^{1/\alpha}\right)+o(1)}{
\exp\left(t\right)+o(1)} \\ &=
\exp\left(-\frac{1}{\alpha}\ell^2 m^2 t^{1-2\beta}\right)+o(\exp(-t))\qquad\text{as $t\to+\infty$}.
\end{split}
\]
Taking $\beta<1/2$ we obtain
\[
\frac{E_\alpha\left(t^{\alpha}(1-\ell^2 m^2 t^{-2\beta})\right)}{
E_\alpha\left(t^{\alpha}\right)} = o(1)\qquad\text{as $t\to+\infty$}.
\]
Analogously,
\[
t^\alpha\frac{E_\alpha\left(t^{\alpha}(1-A_\pm^2 m^2 t^{-2\beta})\right)}{
E_\alpha\left(t^{\alpha}\right)} = t^\alpha\exp\left(-\frac{1}{\alpha}A_\pm^2 m^2 t^{1-2\beta}\right)+t^\alpha o(\exp(-t)) = o(1)\quad\text{as $t\to+\infty$},
\]
and this establishes Theorem \ref{prop:N=1}.
\end{proof}

\section*{Acknowledgments}
S. D. and E. V. are supported by
the Australian Research Council Discovery Project 170104880 {\em NEW - Nonlocal
Equations at Work}. S. D. is supported by
the DECRA Project
DE180100957 {\em PDEs, free boundaries
and applications} and by the Fulbright Foundation.
G. V. is partially supported by the project ERC Advanced Grant 2013-339958
{\em Complex Patterns for Strongly Interacting Dynamical Systems - COMPAT}. G. V. and B. P. are also partially supported by the PRIN Grant 2015KB9WPT
{\em Variational methods, with applications to problems in mathematical physics and geometry}.
The authors are members of INdAM-GNAMPA.

Part of this work was written on the occasion of a visit of B. P. and G. V.
to the University of Melbourne, and of G. V. to the University of Western Australia.

\bibliographystyle{alpha}
\newcommand{\etalchar}[1]{$^{#1}$}

\end{document}